\documentclass[12pt,a4paper]{amsart}
\usepackage{amsmath,amscd}
\usepackage{latexsym,amsmath,amssymb,mathrsfs,setspace,enumerate,tikz}
\usepackage{subcaption}
\usepackage[T1]{fontenc}
\usepackage[utf8]{inputenc}
\usepackage{lmodern}
\usepackage{amssymb}
\usepackage{multicol}
\usepackage{tikz-network}
\usetikzlibrary{automata,positioning}
\usepackage[english]{babel} 
\usepackage{blindtext}

\newtheorem{prop}{Proposition}
\newtheorem{thm}{Theorem}
\newtheorem{rem}{Remark}

\newtheorem{exam}{Example}

\newtheorem{dfn}{Definition}

\title[Group lattices over division rings]{Group lattices over division rings\\ \normalfont{\tiny{A tribute to the renowned Indian Mathematician K.S.S.Nambooripad}}}
\author {P.G. Romeo$^1$ and Alanka Thomas$^2$}
\address{$^1$ Professor, Department of Mathematics, Cochin University of Science and Technology, Kochi, Kerala, INDIA \newline $^{2}$ Research Scholar, Department of Mathematics, Cochin University of Science and Technology, Kochi, Kerala, INDIA.  }
\email{$romeo_-parackal@yahoo.com,\, alankathomas1@gmail.com $}
\thanks{Second author wishes to thank Council of Scientific and Industrial Research(CSIR) INDIA, for providing financial support under JRF scheme through CSIR-UGC NET}
\usepackage{graphicx}
\graphicspath{ {./images/} }
\begin{document}
\begin{abstract}
Group action on an algebraic structure gives a representation of the group by automorphisms on the structure. An action of a group on a lattice gives a group lattice ($G$-lattice). K.S.S. Nambooripad introduced $G$-lattice in 1990 (cf.\cite{kss}). This paper revisits Nambooripad's construction of $G$-lattice.

\end{abstract}
\maketitle

Prof.K.S.S. Nambooripad (1935-2020) is one of the most celebrated Indian mathematicians who has made significant contributions to semigroup theory. Nambooripad started his research at the University of Kerala, India, in 1965. He submitted a brilliant doctoral dissertation on the structure theory of regular semigroups to the University of Kerala in 1973 and was awarded the doctoral degree in 1974. His thesis was published as a Memoir of the American Mathematical society in 1979 (cf.\cite{kssn}). Nambooripad's contributions to semigroup theory research were immense. He published outstanding research papers in esteemed journals such as Semigroup Forum, Journal of algebra, Proceedings of the Edinburgh Mathematical Society, Pacific Journal, and the like. Professor Nambooripad was a champion of the free software movement and was instrumental in popularizing Latex software in India. He also played a vital role in establishing an internationally reputed school on semigroup theory research in Kerala. 

Nambooripad's significant contributions were on the structure theory of regular semigroups, and he described the structure of regular semigroups in two ways. In the first method, he inspected the structure of a regular semigroup using the notions of biordered sets and inductive groupoids (cf.\cite{kssn}). In the second way, his approach was categorical, where he described the category of principal left and right ideals of regular semigroups, which he called normal categories, and certain categorical relations called cross-connections (cf.\cite{kssna}). By using these normal categories and cross-connections, he again described the structure of regular semigroups in a more abstract (categorical) fashion.

Later, at the beginning of the '90s, Nambooripad studied the action of groups on lattices. He called such a structure a group lattice. The results regarding this appeared in the proceedings of the Monash conference on semigroup theory organized in honor of G.B. Preston in 1990 (cf.\cite{kss}).

Now let us recall that an action of a group on a set is a map from $G\times X\rightarrow X$ which takes $(g,x)$ to $gx$ satisfies the following properties, 
\begin{enumerate}
\item $ex=x$ for all $x\in X$
\item $g(hx)=(gh)x$ for all $g,h\in G$ and $x\in X.$
\end{enumerate}
Then $X$ is called a $G-$set. An action of $G$ on $X$ gives a homomorphism that assigns a bijection on $X$ to each $g$ in $G$. In general, an action of a group on an algebraic structure gives a representation of $G$ by automorphisms of the respective algebraic structure. Nambooripad studied the action of a group on a lattice, and he called the resulting structure a group lattice. This paper is motivated by the mathematical work of Nambooripad (cf.\cite{meakin}). It is a revisit to the group lattice, introduced by Nambooripad, in which we obtained some modifications on results and proofs. We dedicate this to the memory of K.S.S. Nambooripad, who was the thesis adviser of the first author.

We have three sections in this paper. The first section gives the definition and examples of $G-$lattice. The second section is devoted to discussing $G$-lattices over arguesian geomodular lattices, the lattice of submodules of a module $V$ over a division ring $K$ with $rank(V)\geq 3$. If $L\cong L(V)$, $L$ is said to be coordinatized by $V$. If $G$ acts on an arguesian geomodular lattice coordinatized by a module $V$ over $K$, it is called a $G-$lattice over $K$. Nambooripad initiated the study of $G-$lattices over $K$. He established a one-one correspondence between $G-$lattices over $K$ and representations of $G$ over $K$ with rank at least $3$. For the third section, one should have the notion of the Schreier extension. If $G$ and $N$ are two groups, the Schreier extension of $N$ by $G$ is a group $H$, having $N$ as a normal subgroup and $H/N\cong G$. Nambooripad used the concept of Schreier extension (cf.\cite{hall}) of $K$ over $G$ while studying $G$-lattices over $K$ and observed correspondence between Schreier extensions of $K$ by $G$ and $G-$lattices over $K$. The third section studies Schreier extensions and their correspondence with group lattices over a division ring.
\section{Group lattices}
An action of a group on an algebraic structure is crucial as it gives a representation of the group by automorphisms on the structure, in particular a group acting on a lattice gives rise to a group lattice. 
\begin{dfn}\label{dfn1}
Let $G$ be a group and $(L,\leq,\wedge,\vee)$ be a lattice. An action of $G$ on $L$ is defined as $gm\in L$ where $g\in G$ and $m\in L$ which satisfies the following, for each $g,h \in G$ and $m,n\in L$
\begin{enumerate}
\item $g(hm)=(gh)m$
\item $em=m$ where $e$ is the identity in $G$
\item $m\leq n$ if and only if $gm\leq gn$
\item $g(m\wedge n)=gm\wedge gn$
\item $g(m\vee n)= gm\vee gn$.
\end{enumerate}
\end{dfn}
If a group $G$ acts on $L$, $L$ is called a $G-lattice$. A $group \, lattice$ is a $G$-lattice for some group $G$. A $G-sublattice$ $L'$ of a G-lattice $L$ is a sublattice of $L$ which is itself a G-lattice under the same product as that of $L$. For a $G-$sublattice $L'$, $gm\in L'$ for all $g\in G$ and $m\in L'$ and the properties $(1)$ to $(5)$ of Definition:\ref{dfn1} inherit from $L$.

\begin{exam}
Let $G$ be a group and $L(G)$ be the lattice of all subgroups of $G$. Then $L(G)$ is a $G-lattice$ under conjugation, defined by \\$gH=gHg^{-1}=\{ghg^{-1}\mid h\in H\}$ for all $g\in G$ and $H\leqslant G$.
\end{exam}
The sublattice $LN(G)$ of all normal subgroups of $G$ is a $G$-sublattice of $L(G)$. Since, the product $gH=gHg^{-1}=H\in LN(G)$ for all $H\in LN(G)$ and $g\in G$.
\begin{exam}
If $X$ is a $G-set$, the power set $\mathscr{P}(X)$ is a $G-$lattice with respect to the product $gA=\{ga\mid a\in A\}$ for all $g\in G$ and $A\subseteq X$.
\end{exam}
Before moving to more examples, recall some preliminary results required. 
\begin{dfn}
Let $V$ be a module over a division ring $K$. A semilinear transformation $f$ on $V$ is a map $f:V\rightarrow V$ such that,
\begin{enumerate}
\item $f(u+v)=f(u)+f(v)$ for all $u,v \in V$,
\item there is an automorphism $\theta_f$ on $K$ such that $f(\alpha v)=\theta_f(\alpha)f(v)$ for all $v\in V$ and $\alpha\in K$.
\end{enumerate}
\end{dfn}
A bijective semilinear transformation is called a semilinear automorphism. The set of all semilinear automorphisms on $V$ form a group under composition, denoted by $SGL(V)$. A linear transformation $f$ on $V$ is a map $f:V\rightarrow V$ such that,
\begin{enumerate}
\item $f(u+v)=f(u)+f(v)$ for all $u,v \in V$,
\item $f(\alpha v)=\alpha f(v)$ for all $v\in V$ and $\alpha\in K$.
\end{enumerate}
Linear automorphisms are bijective linear transformations, and the set of all linear automorpisms form a group under composition, denoted by $GL(V)$. Since, every linear transformation is a semilinear transformation having $\theta_f$ as the identity automorphism on $K$, the group $GL(V)$ is a subgroup of $SGL(V)$.
\begin{dfn}\label{dfn4}
A semilinear projective representation of a group $G$ is a map $\rho : G\rightarrow SGL(V)$ such that for all $g,h\in G$, there exist $\alpha(g,h)\in K^*$ such that $\rho(g)\rho(h)=\alpha(g,h)\rho(gh)$.
\end{dfn}
The map $\rho$ definied above is not a group homomorphism in general, but when $\alpha(g,h)=1$ for all $g,h\in G$, $\rho$ is a group homomorphism and the representation is called a semilinear representation of the group $G$.

A projective linear representation of $G$ is a semilinear projective representation for which, each $\rho(g)$ is a linear automorphism. Here the codomain of $\rho$ is restricted to the subgroup $GL(V)$. 
A linear representation of a group $G$ is a map $\rho : G\rightarrow GL(V)$ such that $\rho(g)\rho(h)=\rho(gh)$ for all $g,h\in G$. Clearly, a linear representation is both projective and semilinear.
\begin{dfn}
Two semilinear projective representations $\rho$ and $\tilde{\rho}$ of a group $G$ are equivalent, if there exists a map $\eta:G\rightarrow K^*$ such that $\tilde{\rho}(g)=\eta(g)\rho(g)$. 
\end{dfn}
For the remaining part of the paper, a representation of $G$ over $K$ means a semilinear projective representation of $G$ on a module $V$ over $K$.
\begin{exam}\label{exam3}
Let $V$ be module over a division ring $K$ and $\rho: G \rightarrow SGL(V)$ be a representation of $G$ over $K$. Then, the lattice of submodules $L(V)$ of $V$, is a $G-$lattice with respect to the product defined as, $gW= \rho(g)W=\{\rho(g)w\mid w\in W\}$ for all $g\in G$ and $W\in L(V)$. 
\end{exam}
The $G$-lattice $L(V)$ defined above is called a semilinear projective $G$-lattice. 
Similarly  we have linear, projective and semilinear $G$-lattices on $L(V)$, according to the choice of the representations. 
\par It is well known that every group action provides a representation of the group and vice versa, we can observe the same for group lattices also. For that, consider a representation $\rho$ of $G$ by automorphisms on $L$, which is a homomorphism from $G$ to $AutL$, the group of all lattice automorphisms on $L$. Define the product $gm=\rho(g)m$ for all $g\in G$ and $m\in L$. Since $\rho$ is a group homomorphism, the product satisfies properties $(1)$ and $(2)$ in Definition: \ref{dfn1} and the rest follows from the fact that $\rho(g)$ is a lattice automorphism for each $g\in G$. Hence $L$ is a $G-$lattice with respect to this action.
\par For the converse part, consider a $G-$lattice $L$. For each $g\in G$ define  $\rho(g):L\rightarrow L$ such that $\rho(g)m=gm$. Properties $(3)$ to $(5)$ in Definition: \ref{dfn1} implies that $\rho(g)$ is a lattice homomorphism. $\rho(g)m=\rho(g)n$ implies $gm=gn$ and so $m=(g^{-1}g)m=g^{-1}(gm)=g^{-1}(gn)=(g^{-1}g)n=n$. So $\rho(g)$ is injective. $\rho(g)$ is surjective, since $g^{-1}m$ is the preimage of $m$ under $\rho(g)$. Hence $\rho(g)$ is a lattice automorphism for each $g\in G$. From property $(1)$ of Definition: \ref{dfn1}, $\rho(gh)=\rho(g)\rho(h)$ and so the map $\rho: G\rightarrow AutL$ which takes each $g\in G$ to $\rho(g)$ is a group homomorphism. ie., $\rho$ is a representation of $G$ by automorphisms on $L$. Hence there is a one-one correspondence between the $G-$lattices and the representations of $G$ by lattice automorphisms. Thus we can summarize these observations as the following theorem.
\begin{thm}\label{thm1}
Let $G$ be a group and $L$ be a lattice. If $\rho: G\rightarrow AutL$ is a representation of $G$ by automorphisms on $L$, for each $g\in G$ and $m\in L$, the product $gm=\rho(g)m$ is an action of $G$ on $L$ so that $L$ is a $G$-lattice. Conversely if $L$ is a $G$-lattice, the map $\rho: G\rightarrow AutL$ such that $\rho(g)m= gm$ is a representaton of $G$ by automorphisms on $L$.
\end{thm}

\section{G-lattices over a division ring}
Next we focus on group lattices over a division ring. Let us first familiarize the terms and notations used in the sequel. Here a division ring is denoted by $K$, and we use $' V'$ for a module over $K$. The set $L(V)$ of all submodules of $V$ form a lattice, and $L(V)$ is said to be coordinatized by $V$. If $rank(V)\geq 3$, $L(V)$ is called an arguesian geomodular lattice. If a group $G$ acts on an arguesian geomodular lattice, it is called a $G$-lattice over $K$, more generally a group lattice over the division ring $K$.
A lattice morphism $f:L(V)\rightarrow L(V)$ is coordinatizable if there exists a semilinear transformation $f':V\rightarrow V$ such that $fW=f'W=\{f'w\mid w\in W\}$. We assume, all lattice morphisms between arguesian geomodular lattices considered here are coordinatizable.

The following theorem establishes a one-one correspondence between $G$-lattices over $K$ and semilinear projective representations of $G$ over $K$.
\begin{thm}\label{thm2}
Let $G$ be a group and $K$ be a division ring. Every representation of $G$ over $K$ determines a $G$-lattice over $K$. Conversely, given a $G$-lattice $L$ over $K$, there is a representation $\rho$ of $G$ over $K$ that determines $L$.
\end{thm}
\begin{proof}
Let $\rho:G\rightarrow SGL(V)$ be a representation of $G$ over $K$. Then $L(V)$ is a $G-$lattice with respect to the product $gW=\rho(g)W$ for all $g\in G$ and $W\in L(V)$.\\
Conversely, $L$ be a $G$-lattice, there exist a module $V$ over $K$ having $rank(V)\geq 3$ such that $L\cong L(V)$. By Theorem: \ref{thm1} there exist a representation $\rho:G \rightarrow AutL$ determines the $G-$lattice $L$. Also there is a semilinear automorphism $\tilde{\rho}(g):V\rightarrow V$ that coordinatizes $\rho(g):L\rightarrow L$, and 
$gW=\rho(g)W=\tilde{\rho}(g)W$ for all $g\in G$ and $W\in L$. 
Let $\langle v \rangle$ be the submodule generated by $v\in V$, then 
$$\tilde{\rho}(g)\tilde{\rho}(h)\langle v\rangle = \rho(g)\rho(h)\langle v\rangle= \rho(gh)\langle v\rangle=\tilde{\rho}(gh)\langle v\rangle $$
 and so $\langle\tilde{\rho}(g)\tilde{\rho}(h)v\rangle=\langle\tilde{\rho}(gh)v\rangle$. Then there exist a $\alpha_v\in K^*$ such that $\tilde{\rho}(g)\tilde{\rho}(h)v= \alpha_v  \tilde{\rho}(gh)v$. Let $w\in V$ be linearly independant with $v$. Since $\tilde{\rho}(gh)$ is an automorphism, $\tilde{\rho}(gh)v$ and $\tilde{\rho}(gh)w$ are also linearly independant. Consider 
\[ \begin{array}{lcl}
\mbox{$\alpha_{v+w}\tilde{\rho}(gh)v+\alpha_{v+w}\tilde{\rho}(gh)w$} & = & \mbox{$\alpha_{v+w}\tilde{\rho}(gh)(v+w)
   = \tilde{\rho}(g)\tilde{\rho}(h)(v+w)$} \\
 & = & \mbox{$\tilde{\rho}(g)\tilde{\rho}(h)v + \tilde{\rho}(g)\tilde{\rho}(h)w$} \\
 & = & \mbox{$\alpha_v  \tilde{\rho}(gh)v+\alpha_w  \tilde{\rho}(gh)w$} 
 \end{array}\] 
then $\alpha_v=\alpha_{v+w}=\alpha_w .$
  Similarly for any scalar $\beta$ and $w$ linearly independant with $\beta v$ we get $\alpha_v=\alpha_{\beta v}=\alpha_w$. So the scalar $\alpha=\alpha_v$ is independant of the choice of $v$. Hence $\tilde{\rho}(g)\tilde{\rho}(h)= \alpha \tilde{\rho}(gh)$ and $\tilde{\rho}:G\rightarrow SGL(V)$ is a representation of $G$ over $K$.
\end{proof}
\begin{thm}
If $\rho$ and $\tilde{\rho}$ are two representations of $G$ over $K$, then they determine the same $G-$lattice if and only if they are equivalent.
\end{thm}
\begin{proof}
If $\rho$ and $\tilde{\rho}$ are equivalent, there exist $\eta : G\rightarrow K^*$ such that $\tilde{\rho}(g)=\eta(g)\rho(g)$. Then $\tilde{\rho}(g)W=\eta(g)\rho(g)W=\rho(g)W$ for all $g\in G$ and $W\in L(V)$ and so $\rho$ and $\tilde{\rho}$ induce the same $G-$lattice.

Conversely if $\rho$ and $\tilde{\rho}$ induce the same $G-$lattice, then for any $v\in V$, $\langle\rho(g)v\rangle=\rho(g)\langle v\rangle=\tilde{\rho}(g)\langle v\rangle= \langle\tilde{\rho}(g)v\rangle$. 
Similarly as in the proof of Theorem: \ref{thm2}, there exist $\alpha_g \in K^*$ independant of $v$ such that 
$\tilde{\rho}(g)v=\alpha_g\rho(g)v$ and so $\tilde{\rho}(g)=\alpha_g\rho(g)$. $\eta:G\rightarrow K^*$ such that $\eta(g)=\alpha_g$ satisfies the equation $\tilde{\rho}(g)=\eta(g)\rho(g)$ and so $\rho$ and $\tilde{\rho}$ are equivalent.
\end{proof}

\section{Schreier extension and associated Group lattices}
Let $N$ and $G$ be groups. A Schreier extension of $N$ by $G$  is a group $H$ having $N$ as a normal subgroup and $H/N\cong G$ (cf.\cite{hall}). Let $H$ be a group having $N$ as a normal subgroup and $H/N\cong G$. Then 
$H=\bigcup\limits_{g\in G} N\overline{g}=\{a\overline{g}\mid a\in N \text{ and } g\in G\}$, where $\{\overline{g}\mid g\in G\}$ is the set of all representatives from each coset of $N$. For each $g\in G$ the map $b\rightarrow \overline{g}b\overline{g}^{-1}$ is a group automorphism in $N$ and let it be denoted by $\chi(g)$. Also since $\overline{gh}$ is the representative of the coset containing $\overline{g}\overline{h}$ there exist $[g,h]\in N$ such that $\overline{g}\overline{h}=[g,h]\overline{gh}$. Then $(a\overline{g})(b\overline{h})=a\overline{g}b\overline{g}^{-1}\overline{g}\overline{h}=a\chi(g)(b)[g,h]\overline{gh}$ and $H$ is a group with respect to this binary operation. Hence we have functions,
\begin{enumerate}
\item $\chi : G\rightarrow AutN$ defined by $\chi(g)(a)=\overline{g}a\overline{g}^{-1}$
\item $[_-,_-]:G\times G\rightarrow N$ such that $\overline{g}\overline{h}=[g,h]\overline{gh}$.
\end{enumerate}
and we can see that $\chi$ and $[_-,_-]$ has the following properties
\begin{itemize}
\item[$E1)$] $\chi(g)\chi(h)=[g,h]\chi(gh)[g,h]^{-1}$
\item[$E2)$] $[g,h][gh,k]=\chi(g)([h,k])[g,hk]$
\item[$E3)$] $[1,1]=1.$
\end{itemize}
\par In general, if $\chi : G\rightarrow AutN$ and $[_-,_-]:G\times G\rightarrow N$ are function satisfying above properties, $H=\{a\overline{g}\mid a\in N \text{ and } g\in G\}$ is a group with respect to the binary operation $(a\overline{g})(b\overline{h})=a\chi(g)(b)[g,h]\overline{gh}$. This is the construction of Schreier extension of $N$ by $G$. The pair of functions $(\chi,[_-,_-])$ called a factor system for $(N,G)$ and the extension is denoted by $H(\chi,[_-,_-])$. Factor systems $(\chi,[_-,_-])$ and $(\chi ',[_-,_-]')$ are equivalent, if there exists a map $\mu :G\rightarrow N$ such that 
\begin{enumerate}
\item[$E4)$] $\chi '(g)=\mu(g)^{-1}\chi(g)\mu(g)$
\item[$E5)$] $[g,h]\mu(gh)=\mu(g)\chi '(g)(\mu(h))[g,h]'$
\item[$E6)$] $\mu(1)=1$.
\end{enumerate}
If $(\chi,[_-,_-])$ and $(\chi ',[_-,_-]')$ are equivalent, the map $a\overline{g}\rightarrow a\mu(g)\tilde{g}$ from $H(\chi,[_-,_-])$ to $H(\chi ',[_-,_-]')$ is an isomorphism. Two factor systems are equivalent if and only if corresponding extensions are isomorphic.
\begin{dfn}
A central extension of $N$ by $G$ is an extension $H(\chi,[_-,_-])$ for which, $[g,h]$ are in the center of $N$, for each $g,h$ in G. 
\end{dfn}
Using $(E1)$, it is seen that the map $\chi: G\rightarrow AutN$ is a homomorphism for a central extension. If $N$ is an abelian group, every Schreier extension of $N$ is a central extension. 

If $\chi(g)=I_N$ for all $g\in G$, the binary operation in $H$ reduces to $(a\overline{g})(b\overline{h})=ab[g,h]\overline{gh}$ and $H$ is called a projective extension of $N$ by $G$. If $[g,h]=1$ for all $g,h\in G$, then $H$ is the semidirect product of $N$ and $G$, this type of extensions are called split extensions or semidirect products. It is easy to see that projective extensions and semidirect products are central extensions.  The direct product of $N$ and $G$ is the Schreier extension of $N$ by $G$ which is both projective and split.
\par Here we deal with Schreier extension of $K^*$ by $G$, where $K^*$ is the group of nonzero elements of a division ring under multiplication. We shall say that $H(\chi,[_-,_-])$ is an extension of $K$ by $G$ and that $(\chi,[_-,_-])$ is a factor system for $(K,G)$, if it is a factor system for $(K^*,G)$ with the property that $\chi(g)\in AutK$ for each $g\in G$. ie., every representation of $G$ over $K$ gives a Schreier extension of $K$ by $G$. For a semilinear projective representation $\rho$, each $\rho(g)$ is a semilinear transformation. So, there exist automorphism $\theta_{\rho(g)}$ on $K$ such that $\rho(g)(\alpha v)=\theta_{\rho(g)}(\alpha)\rho(g)(v)$. For brevity, let it be denoted by $\chi(g)$. Since the representation is projective, for each $g,h\in G$ there exist $[g,h]\in K^*$ such that $\rho(g)\rho(h)=[g,h]\rho(gh)$. We will prove in Proposition: \ref{prop1} that this $\chi$ and $[_-,_-]$ together form a factor system for an extension of $K$ by $G$.
\begin{prop}\label{prop1}
Let $\rho$ be a representation of $G$ over $K$. Then $(\chi,[_-,_-])$ is a factor system for $(K,G)$. 
\end{prop}
\begin{proof}
Given $\rho$ is a representation of $G$ over $K$. For each $g\in G$ there exist an automorphism $\chi(g)$ of $K$ such that $\rho(g)(\alpha v)=\chi(g)(\alpha)\rho(g)(v)$ and $\rho(g)\rho(h)=[g,h]\rho(gh)$ for all $g,h\in G$. $(\chi, [_-,_-])$ is a factor system if it satisfies properties $(E1)-(E3)$.
\[ \begin{array}{lcl}
\mbox{$\chi(g)(\chi(h)(\alpha))\rho(g)\rho(h)(v)$} & = & \mbox{$\rho(g)\rho(h)(\alpha v)=[g,h]\chi(gh)(\alpha)\rho(gh)(v)$} \\
 & = & \mbox{$[g,h]\chi(gh)(\alpha)[g,h]^{-1}[g,h]\rho(gh)(v)$} \\
 & = & \mbox{$[g,h]\chi(gh)(\alpha)[g,h]^{-1}\rho(g)\rho(h)(v)$} 
 \end{array}\] 
Hence $\chi(g)\chi(h)=[g,h]\chi(gh)[g,h]^{-1}$ for all $g,h\in G$.
\[ \begin{array}{lcl}
 \mbox{$[g,h][gh,k]\rho(ghk)(v)$}  & = & \mbox{$[g,h]\rho(gh)(\rho(k)(v) = \big(\rho(g)\rho(h)\big)\rho(k)(v)$}\\
  & = & \mbox{$\rho(g)\big(\rho(h)\rho(k)\big)(v) = \rho(g)\big([h,k]\rho(hk)\big)(v)$}\\
  & = & \mbox{$\chi(g)([h,k])\rho(g)\rho(hk)(v)$}\\
  & = & \mbox{$\chi(g)([h,k])[g,hk]\rho(ghk)(v)$}.
 \end{array}\] 

Hence $[g,h][gh,k]\rho(ghk)(v)=[g,hk]\chi(g)[h,k]\rho(ghk)(v)$ for $g,h,k\in G$ and $v\in V$ and so $[g,h][gh,k]=[g,hk]\chi(g)[h,k]$. It is clear that $[1,1]=1$. Hence $(\chi,[_-,_-])$ is a factor system for $(K,G)$.
\end{proof}
\par A representation $\rho$ of $G$ over $K$ is said to be associated with a factor system $(\chi,[_-,_-])$, if $\chi(g)=\theta_{\rho(g)}$ and $\rho(g)\rho(h)=[g,h]\rho(gh)$ for each $g,h\in G$. A representation is projective if and only if the Schreier extension associated with it is projective. Similarly, a representation is semilinear if and only if the Schreier extension associated with it is a semidirect product. So, direct product $K\times G$ is the Schreier extension of $K$ by $G$ associated with a linear representation of $G$ over $K$. Equivalent representations always give equivalent factor systems. Proposition: \ref{prop2} deals with this observation. 
\begin{prop}\label{prop2}
Let $(\chi,[_-,_-])$ and $(\chi ',[_-,_-]')$ be factor systems for $(K,G)$ and $\rho$ be the representation of $G$ over $K$ associated with $(\chi,[_-,_-])$. Then $(\chi,[_-,_-])$ is equivalent to $(\chi ',[_-,_-]')$ if and only if  there is a representation of $\tilde{\rho}$ associated with $(\chi ',[_-,_-]')$ which is equivalent to $\rho$.
\end{prop}
\begin{proof}
Let $\rho$ and $\tilde{\rho}$ be representation of $G$ over $K$ associated with factor systems $(\chi,[_-,_-])$ and $(\chi ',[_-,_-]')$ for $(K,G)$ respectively. If $\rho$ is equivalent to $\tilde{\rho}$, there exist a map $\mu :G\rightarrow K^*$ such that $\rho(g)=\mu(g)\tilde{\rho}(g)$. 
Then $\mu$ is an equivalence of the factor system $(\chi,[_-,_-])$ with $(\chi ',[_-,_-]')$; for, consider,
\begin{center}
$\mu(g)\chi'(g)(\alpha)\tilde{\rho}(g)(v)=\mu(g)\tilde{\rho}(g)(\alpha v)=\rho(g)(\alpha v)=\chi(g)(\alpha)\mu(g)\tilde{\rho}(g)(v)$. 
\end{center}
So $\mu(g)\chi'(g)=\chi(g)\mu(g)$ for all $g\in G$ and this proves $(E4)$. Now consider 
\[ \begin{array}{lcl}
 \mbox{$[g,h]\mu(gh)\tilde{\rho}(gh)(v)$}  & = & \mbox{$\rho(g)\rho(h)(v)=\mu(g)\tilde{\rho}(g)\mu(h)\tilde{\rho}(h)(v)$}\\
  & = & \mbox{$\mu(g)\chi'(g)(\mu(h))[g,h]'\tilde{\rho}(gh)(v)$}
 \end{array}\] 
Hence $[g,h]\mu(gh)=\mu(g)\chi'(g)(\mu(h))[g,h]'$ and this proves $(E5)$. $\mu$ satisfies $(E4)-(E6)$ and so $\mu$ is an equivalence between the factor systems $(\chi,[_-,_-])$ and $(\chi ',[_-,_-]')$.\\
Conversely let $(\chi,[_-,_-])$ and $(\chi ',[_-,_-]')$ be equivalent and $\rho$ be a representation of $G$ over $K$ associated with $(\chi,[_-,_-])$. There exist an equivalence $\mu:G\rightarrow K^*$ from $(\chi',[_-,_-]')$ and $(\chi,[_-,_-])$ satisfying $(E4)-(E6)$. For each $g\in G$ define $\tilde{\rho}(g)=\mu(g)\rho(g)$, then $\tilde{\rho}$ is a representation of $G$ over $K$. Since each $g\in G$, $\tilde{\rho}(g)$ preserves addition as in the case of $\rho(g)$. Now using $(E4)$ we get\\
\\
$\tilde{\rho}(g)(\alpha v)=\mu(g)\chi(g)(\alpha)\rho(g)(v)=\chi'(g)(\alpha)\mu(g)\rho(g)(v)=\chi'(g)(\alpha)\tilde{\rho}(g)(v)$\\
\\and so $\tilde{\rho}(g)$ is a semilinear transformation for each $g\in G$. Using the bijectivity of $\rho(g)$, $\tilde{\rho}(g)$ is also bijective and so $\tilde{\rho}(g)$ is a semilinear automorphism for each $g\in G$.
\[ \begin{array}{lcl}
\mbox{$\tilde{\rho}(g)\tilde{\rho}(h)(v)$} & = & \mbox{$\mu(g)\rho(g)\mu(h)\rho(h)(v)=\mu(g)\chi(g)(\mu(h))[g,h]\rho(gh)(v)$} \\
 & = & \mbox{$[g,h]'\mu(gh)\rho(gh)(v)=[g,h]'\tilde{\rho}(gh)(v)$}
 \end{array}\] 
 Hence $\tilde{\rho}(g)\tilde{\rho}(h)=[g,h]'\tilde{\rho}(gh)$ and $\tilde{\rho}$ is a representation of $G$ over $K$. It is clear that $\tilde{\rho}$ is associated with the factor system $(\chi',[_-,_-])$. Since $\mu:G\rightarrow K^*$ is a map such that $\tilde{\rho}(g)=\mu(g)\rho(g)$ for all $g\in G$, $\rho$ and $\tilde{\rho}$ are equivalent representations of $G$ over $K$. Hence the proof.
\end{proof}
Theorem: \ref{thm2} and Proposition: \ref{prop1} together tells us that every $G$-lattice over $K$ gives us a Schreier extension of $K$ by $G$. The converse is also true. We need to have the following definition of a twisted group ring to prove that.
\begin{dfn}
Let $H=H(\chi,[_-,_-])$ be an extension of $K$ by $G$. The twisted group ring $K(G;H)$ is the free module generated by the set $\{\bar{g}\mid g\in G\}$ over $K$.
\end{dfn}
For $u= \sum\limits_{g\in G}a_g\bar{g}$ and $v= \sum\limits_{h\in G}b_h\bar{h}$ in $K(G;H)$ and $c\in K$ we have, \\
\begin{center}
$u+v=\sum\limits_{g\in G}(a_g+b_g)\bar{g} \qquad cu=\sum\limits_{g\in G} ca_g\bar{g}$  \\
$uv=\sum\limits_{k\in G}c_k\bar{k}$ where $c_k=\sum\limits_{gh=k}a_g\chi(g)(b_h)[g,h]$
\end{center}
The above product is obtained by distributively extending the product in $H$. $K(G;H)$ is an associative ring and need not be an algebra in general. If factor systems $(\chi,[_-,_-])$ and $(\chi',[_-,_-]')$ are equivalent, $H(\chi,[_-,_-])$ and $H(\chi',[_-,_-]')$ are isomorphic and so $K(G;H)$ and $K(G;H')$ are also isomorphic. Hence $K(G;H)$ is uniquely determined by an extension $H$.
\begin{thm}\label{thm4}
Let $K$ be a division ring and $G$ a group. Every $G-$lattice $L$ over $K$ determines a unique extension of $K$ by $G$. Conversely, if $H$ is an extension of $K$ by $G$, there is a $G-$lattice $L$ over $K$ such that $H$ is determined by $L$.
\end{thm}
\begin{proof}
Corresponding to each $G-$lattice over $K$, we have a representation of $G$ over $K$ and vice versa. So the first part is clear. Let $H=H(\chi,[_-,_-])$ be an extension of $K$ by $G$. For each $g\in G$ define $\rho(g):K(G;H)\rightarrow K(G;H)$ such that for each $v\in V$, $\rho(g)(v)=\overline{g}v$ where the product on the right is the product in $K(G;H)$. For $v\in V$ and $\alpha\in K$, $\rho(g)(\alpha v)=\overline{g}(\alpha v)=\chi(g)(\alpha)\overline{g}v=\chi(g)(\alpha)\rho(g)(v).$ Hence $\rho(g)$ is a semilinear automorphism having $\theta_{\rho(g)}=\chi(g)$. Also 
$$\rho(g)\rho(h)(v)=\overline{g}(\overline{h}(v))=\overline{g}\overline{h}(v)=[g,h]\overline{gh}(v)=[g,h]\rho(gh)(v)$$
and so $\rho:G\rightarrow SGL(V)$ is a representation associated with the factor system $(\chi,[_-,_-])$. The lattice $L(K(G;H))$ is a $G-$lattice with the action induced by $\overline{\rho}$. It is clear that the extension determined by $L$ is the given $H$. 
\end{proof}
The following example illustrate Theorem: \ref{thm4} for the $C_3$-lattice, $L(\mathbb{R}^3)$.
\begin{exam}\label{exam4}
Consider $\mathbb{R}^3$ and the the lattice $L(\mathbb{R}^3)$ of its subspaces. Let $C_3=\{1,a,a^2\}$ be the cyclic group of order $3$ generated by $a$. $L(\mathbb{R}^3)$ is a $C_3$-lattice  with respect to the action defined as follows.
\end{exam}
\begin{enumerate}
\item $1W=W$ for all $W\in L(\mathbb{R}^3).$
\item $aW$ is the subspace of $\mathbb{R}^3$ obtained by shifting the components of each vector in $W$ one position to the right. For example, $W=\{(x,y,z)\in \mathbb{R}^3\mid x+y=z\}$ is a subspace of $\mathbb{R}^3$ and $aW=\{(x,y,z)\in \mathbb{R}^3\mid x+z=y\}$.
\item $a^2W$ is the subspace of $\mathbb{R}^3$ obtained by shifting the components of each vector in $W$ two positions to the right. We have $a^2W=\{(x,y,z)\in \mathbb{R}^3\mid y+z=x\}$ for the above subspace.
\end{enumerate}
\par Let $v=(x,y,z)\in \mathbb{R}^3$ and $\langle v\rangle=\{(kx,ky,kz)\mid k\in \mathbb{R}\}$ be the subspace spanned by $v$. Then $1\langle v\rangle=\langle v\rangle=\langle (x,y,z)\rangle$,$a\langle v\rangle=\{(kz,kx,ky)\mid k\in\mathbb{R}\}=\langle(z,x,y)\rangle$ and $a^2\langle v\rangle=\{(ky,kz,kx)\mid k\in \mathbb{R}\}=\langle(y,z,x)\rangle$. Define $\rho(1)(x,y,z)=(x,y,x),\rho(a)(x,y,z)=(z,x,y)$ and $\rho(a^2)(x,y,z)$ =$(y,z,x)$. It is easy to see that $\rho(g)$ is an automorphism in $\mathbb{R}^3$ for each $g\in C_3$ and $\rho:C_3\rightarrow Aut\mathbb{R}^3$ which takes each $g\in C_3$ to $\rho(g)$ defined above is a group homomorphism. Hence $\rho$ is a linear representation of $C_3$ over $\mathbb{R}$ and is the representation corresponding to the $C_3$-lattice $L(\mathbb{R}^3)$.\\
Now we discuss the Schreier extension associated with $\rho$. Since $\rho$ is a linear representation, $\chi(g)=\theta_{\rho(g)}=I_{\mathbb{R}^*}$ and $[g,h]=1$ for all $g,h\in C_3$. So Schreier extension of $\mathbb{R}^*$ by $C_3$ is $H=\{k\bar{1}\mid k\in \mathbb{R}^*\}\cup\{k\bar{a}\mid k\in \mathbb{R}^*\}\cup\{k\bar{a^2}\mid k\in \mathbb{R}^*\}$ and the binary operation in $H$ is given by $k\bar{g}k'\bar{h}=k\chi(g)(k')[g,h]\bar{gh}=kk'\overline{gh}$. Define a map $\phi:\mathbb{R}^*\times C_3\rightarrow H$ such that $\phi(k,g)=k\bar{g}$. $\phi$ is a group homomorphism since $\phi(k,g)\phi(k',h)=k\bar{g}k'\bar{h}=kk'\overline{gh}=\phi(kk',gh)$. $\phi(k,g)=\phi(k',h)$ means that $k\bar{g}=k'\bar{h}$, so $k=k'$ and $g=h$ which implies $(k,g)=(k',h)$ and hence $\phi$ is injective. $\phi$ is surjective since for $k\bar{g}\in H$ there is always $(k,g)\in \mathbb{R}^*\times C_3$ such that $\phi(k,g)=k\bar{g}$. Hence $H$ is isomorphic to $\mathbb{R}^*\times C_3$. \\
The twisted group ring $\mathbb{R}(C_3,\mathbb{R}^*\times C_3)
= \{x\bar{1}+y\bar{a}+z\bar{a^2}\}$ is a vector space with respect to the addition and scalar multiplication defined by $$\big(x_1\bar{1}+y_1\bar{a}+z_1\bar{a^2}\big)+\big(x_2\bar{1}+y_2\bar{a}+z_2\bar{a^2}\big)=(x_1+x_2)\bar{1}+(y_1+y_2)\bar{a}+(z_1+z_2)\bar{a^2}$$ $$k\big(x\bar{1}+y\bar{a}+z\bar{a^2}\big)=kx\bar{1}+ky\bar{a}+kz\bar{a^2}$$ Then $\theta:\mathbb{R}^3\rightarrow \mathbb{R}(C_3,\mathbb{R}^*\times C_3)$ given by $\theta(x,y,z)= x\bar{1}+y\bar{a}+z\bar{a^2}$ is an isomorphism. So $L(\mathbb{R}^3)\cong L(\mathbb{R}(C_3,\mathbb{R}^*\times C_3))$. Note that $L(\mathbb{R}(C_3,\mathbb{R}^*\times C_3))$ is a $C_3$-lattice with respect to the product defined as,
\begin{enumerate}
\item $1W=W$
\item $aW=\{\bar{a}(x\bar{1}+y\bar{a}+z\bar{a^2})\mid x\bar{1}+y\bar{a}+z\bar{a^2}\in W\}=\{z\bar{1}+x\bar{a}+y\bar{a^2}\mid x\bar{1}+y\bar{a}+z\bar{a^2}\in W\}$
\item $a^2W=\{\bar{a^2}(x\bar{1}+y\bar{a}+z\bar{a^2})\mid x\bar{1}+y\bar{a}+z\bar{a^2}\in W\}=\{y\bar{1}+z\bar{a}+x\bar{a^2}\mid x\bar{1}+y\bar{a}+z\bar{a^2}\in W\}$
\end{enumerate}
Hence $C_3$-lattices $L(\mathbb{R}(C_3,\mathbb{R}^*\times C_3))$ and $L(\mathbb{R}^3)$ are isomorphic.   
\begin{rem}
$V$ need not be isomorphic with $K(G;H)$ in general. If we replace $C_3$ by $S_3$ in the above example and define an action of $S_3$ on $L(\mathbb{R}^3)$ as above, then $\mathbb{R}(S_3;\mathbb{R}^*\times S_3)$ is isomorphic with $\mathbb{R}^6$ but not with $\mathbb{R}^3$.
\end{rem}
\begin{thm}
Let $G$ be a group and $K$ be a division ring. If $H$ is the extension determined by the $G-$lattice $L$ over $K$,
\begin{enumerate}
\item $H$ is a projective extension if and only if $L$ is a projective $G$-lattice
\item $H$ is a semidirect product if and only if $L$ is a semilinear $G$-lattice
\item $L$ is a linear $G$-lattice if and only if $H$ is a linear extension and so $H\cong K^*\times G$.
\end{enumerate}
\end{thm}
\begin{proof}
 We can see that $L$ is a projective $G$-lattice if and only if corresponding representation $\rho$ of $G$ over $K$ is projective, that is if and only if the extension associated to $\rho$ is projective. Similarly, $L$ is a semilinear $G$-lattice if and only if corresponding representation $\rho$ of $G$ over $K$ is semilinear, that is if and only if the extension associated to $\rho$ is a semidirect product. If $H$ is the direct product $ K^*\times G$,  $H$ is a projective as well as a split extension of $K$ by $G$, and the representation $\rho$ associated with $H$ is both projective and semilinear, which is the linear representation of $G$ over $K$. So the $G$-lattice $L$ induced by $\rho$ is the linear $G$-lattice.
\end{proof}
The following theorem gives a necessary and sufficient condition so that the twisted group ring is an algebra.
\begin{thm}\label{thm6}
Let $L$ be a $G-$lattice over $K$ and $H$ be the extension determined by $L$. Then, the twisted group ring $K(G;H)$ is an algebra over $K$ if and only if $L$ is a projective $G-$lattice over the field $K$.
\end{thm}
\begin{proof}
Assume that $L$ is a projective $G-$lattice over the field $K$. To show that $K(G;H)$ an algebra it remains to prove that $(au)v=u(av)=a(uv)$ for all $a\in K$ and $u,v\in K(G;H)$. Since the product in $K(G;H)$ is obtained by distributively extending the product in $H$, it is enough to check the above property for the elements in the basis $\{\overline{g}\mid g\in G\}$ of $K(G;H)$. $\overline{g}(a\overline{h})=a(\overline{g}\overline{h})=(a\overline{g})\overline{h}$ since $H$ is a projective extension, this proves that $K(G;H)$ is a $K-$algebra.\\
Conversely let $K(G;H)$ be an algebra. Then $ba\overline{1}=(b\overline{1})(a\overline{1})=a(b\overline{1})=ab\overline{1}$ which implies that $ab=ba$ for all $a,b\in K^*$ and so $K$ is a field. $\chi(g)(a)\overline{g}=(a\overline{1})\overline{g}=a\overline{g}$ and so $\chi(g)(a)=a$ for all $a\in K^*$ and $g\in G$. Hence $\chi(g)=I_K$ for each $g\in G$ and so $H$ is a projective extension of $K$ by $G$.  
\end{proof}
Since $\rho:G\rightarrow SGL(V)$ is a semilinear projective representation, $gv=\rho(g)(v)$ is not an action of $G$ on $V$ in general. However, it may turn out to be an action by distributively extending the multiplication.

\begin{prop}
Let $G$ be a group, $K$ a division ring and $H$ be a Schreier extension of $K$ by $G$. If the representation $\rho:G\rightarrow SGL(V)$ of $G$ is associated with the extension $H$, then $V$ is a $K(G;H)$ module. Conversely let $M$ be a $K(G;H)$ module. Then $M$ is a module over $K$, for each $g\in G$ and $v\in V$, $\rho(g)(v)=\overline{g}v$ defines a semilinear automorphism $\rho(g)$ on $M$ and $\rho:G\rightarrow SGL(V)$ is a representation which determines the extension $H$.
\end{prop}
\begin{proof}
Assume that $\rho:G\rightarrow SGL(V)$ is a representation of $G$ over $K$.

For $s=\sum\limits_{g\in G}a_g\overline{g}$ and $v\in V$ define $sv=\sum\limits_{g\in G}a_g\rho(g)(v)$. Then $V$ is a $K(G;H)$-module with respect to this product. Consider $s=\sum\limits_{g\in G}a_g\overline{g}$, $t=\sum\limits_{g\in G}b_g\overline{g}$ in $K(G;H)$ and $u,v\in V$. Using the semilinearity of $\rho(g)$ and the $K-$module structure on $V$ we get,
\begin{enumerate}
\item $s(u+v)=\sum\limits_{g\in G}a_g\big(\rho(g)(u)+\rho(g)(v)\big)=su+sv$.
\item $(s+t)v=\sum\limits_{g\in G}a_g\rho(g)(v)+ b_g\rho(g)(v)=sv+tv$.
\item $s(tv)=\sum\limits_{g\in G}a_g\rho(g)\bigg(\sum\limits_{h\in G}b_h\rho(h)(v)\bigg)=\sum\limits_{g\in G}\sum\limits_{h\in G}a_g\chi(g)(b_h)\rho(g)\rho(h)(v)$\\=$\sum\limits_{k\in G}\bigg(\sum\limits_{gh=k}a_g\chi(g)(b_h)[g,h]\bigg)\rho(k)(v)=(st)v$.
\item $\bar{1}v=\rho(1)(v)=v$.
\item $(bs)v=\sum\limits_{g\in G}ba_g\rho(g)(v)=b\sum\limits_{g\in G}a_g\rho(g)(v)=b(sv)$.
\end{enumerate}
Converse follows.
\end{proof}
Thus we summerize that given an extension $H$ of $G$ by $K$, every $K(G;H)$-module $V$ induce a $G-$lattice on $L(V)$ and conversely, every $G-$lattice induces an extension $H$ and a $K(G; H)$-module. Moreover, when  $H$ and $H'$ are isomorphic extensions, $K(G; H)$ and $K(G; H')$ are isomorphic, and they induce the same action on the module $V$ corresponding to the $G-$lattice $L(V)$. 

\end{document}